\documentclass[10pt]{amsart}

\usepackage{color,soul}
\sethlcolor{yellow}
\usepackage[margin=1.25in]{geometry} 
\usepackage{kpfonts}
\usepackage{enumerate}
\usepackage[mathscr]{eucal}
\usepackage{color} 
\usepackage[colorlinks=true, linkcolor=red, citecolor=blue, menucolor=black]{hyperref}

\usepackage{amsmath,amsthm, amsfonts, amssymb,MnSymbol,extpfeil,mathtools}
\numberwithin{equation}{section}
\usepackage{graphicx}
\usepackage{listings}
\usepackage{enumerate}
\usepackage{centernot}
\usepackage{enumitem}
\usepackage{xfrac}
\usepackage{leftindex}
\usepackage{hyperref}
\usepackage{adjustbox}

\usepackage{amsmath, amsthm, amssymb}
\numberwithin{equation}{section}
\usepackage{xspace}
\usepackage{graphicx}
\usepackage{array}
\usepackage{braket}
\usepackage{geometry}
\usepackage{multicol}
\usepackage{mathtools}
\usepackage{enumerate}
\usepackage{delarray}
\usepackage{mathtools}
\usepackage{faktor,extpfeil} 
\usepackage{mathrsfs}
\usepackage{tikz-cd}
\usepackage{hyperref}
\usepackage[normalem]{ulem} 
\usepackage[italicdiff]{physics} 
\usepackage{bbm} 
\usepackage{float}
\usepackage{stmaryrd} 
\usepackage{aligned-overset}
\usepackage{xcolor}
\usepackage{cite}

\theoremstyle{plain}
\newtheorem{thm}{Theorem}[section]
\newtheorem{lemma}[thm]{Lemma}
\newtheorem{cor}[thm]{Corollary}

\newtheorem{maintheorem}{Theorem}

\newtheorem{maincor}[maintheorem]{Corollary}

\newtheorem{prop}[thm]{Proposition}
\theoremstyle{remark}

\theoremstyle{definition}

\newtheorem*{notation*}{Notation}

\DeclarePairedDelimiterX{\inp}[2]{\langle}{\rangle}{#1, #2}

\makeatletter
\newcommand*{\bigcdot}{}
\DeclareRobustCommand*{\bigcdot}{%
  \mathbin{\mathpalette\bigcdot@{}}%
}
\newcommand*{\bigcdot@scalefactor}{.5}
\newcommand*{\bigcdot@widthfactor}{1.15}
\newcommand*{\bigcdot@}[2]{%
  \sbox0{$#1\vcenter{}$}
  \sbox2{$#1\cdot\m@th$}%
  \hbox to \bigcdot@widthfactor\wd2{%
    \hfil
    \raise\ht0\hbox{%
      \scalebox{\bigcdot@scalefactor}{%
        \lower\ht0\hbox{$#1\bullet\m@th$}%
      }%
    }%
    \hfil
  }%
}
\makeatother

\newcommand{\N}{\mathbb{N}\xspace}

\newcommand{\ca}{\curvearrowright}

\DeclareMathOperator{\Prob}{Prob}

\DeclareMathOperator{\id}{id}

\newcommand{\eps}{\varepsilon}

\usepackage{enumitem}
\usepackage{color}
\newlist{steps}{enumerate}{1}
\usepackage[foot]{amsaddr}
\setlist[steps, 1]{label = Step \arabic*:}

\hypersetup{%
  colorlinks=true,%
  linkcolor=blue,%
  citecolor=blue,%
  filecolor=blue,%
  menucolor=blue,%
  urlcolor=blue,%
  pdfnewwindow=true,%
  pdfstartview=FitBH
}   

\title{Bi-exact Wreath-like Product Groups }

\author{Ionu\c{t} Chifan}
\address{The University of Iowa, Mathematics Department, 14 MacLean Hall, Iowa City,
IA, 52240, USA}
\email{ionut-chifan@uiowa.edu}
\author{Zhiyuan Yang}
\address{Purdue University, Department of mathematics,  150 N University St, West Lafayette, IN, 47907, USA}
\email{yang3261@purdue.edu}

\begin{document}

\begin{abstract}
In this short paper, we prove the bi-exactness of all wreath-like product groups $G\in\mathcal{WR}(A,B\curvearrowright I)$, as introduced in \cite{CIOS1}, whenever $A$ is amenable, $B$ is bi-exact, and the action $B\curvearrowright I$ has amenable stabilizers. As applications, we obtain solidity for the associated group von Neumann algebras and, by combining our result with existing rigidity theorems, new examples exhibiting the McDuff superrigidity described in \cite{AMCOS25}.
\end{abstract}

\maketitle

\section{Introduction}

A countable discrete group $G$ is said to be \emph{bi-exact} if it is exact and there exists a map $\mu:G\to \operatorname{Prob}(G)\subset \ell^1(G)
$ such that, for every $h,k\in G$,

\begin{equation}
\lim_{g\longrightarrow\infty}
\bigl\|\mu(kgh)-k\cdot \mu(g)\bigr\|_1=0,
\end{equation}
where $g\to\infty$ means that $g$ eventually leaves every finite subset of $G$, and the left translation action of $G$ on $\operatorname{Prob}(G)$ is given by $(k\cdot\mu)(g)=\mu(k^{-1}g)$ for all $ g,k\in G$. 
\vskip 0.1in
Bi-exactness, introduced by Ozawa,  has emerged as an important group-theoretic manifestation of negative curvature with particularly strong consequences for the associated von Neumann algebras. Most notably, Ozawa showed that the group von Neumann algebra of a bi-exact group is solid \cite{Oza04}, while bi-exactness has subsequently played a central role in unique prime factorization and rigidity results for group von Neumann algebras and crossed products; see, for instance, \cite{OP04,CS13,PV12,CI18,CTY26}---just to enumerate a few. The class is broad enough to contain many of the standard negatively curved groups, including hyperbolic groups and lattices in rank-one simple Lie groups, while at the same time being restrictive enough to impose strong structural properties on their operator algebras. Thus, enlarging the classes of groups known to be bi-exact provides an important avenue for establishing new rigidity and structural results for the associated von Neumann algebras, as will also be illustrated by the applications discussed in this paper.

\vskip 0.1in

Wreath-like product groups were introduced in \cite{CIOS1} as a flexible generalization of classical wreath products and provided, among other things, the first examples of property (T) W$^*$-superrigid groups. These results gave the first substantial positive evidence towards Connes Rigidity Conjecture.
 We briefly recall their definition. 
 
 Given two countable groups $A$, $B$ and an action $B \ca I$ we denote by  $\mathcal {WR}(A,B\ca I)$ the collection of all groups $G$ that can be realized as wreath-like product extensions, i.e.\ there exists a short exact sequence  \begin{equation}\label{wrlext}1\longrightarrow A^{(I)} \hookrightarrow G \overset{\varepsilon}{\twoheadrightarrow} B \longrightarrow 1, \end{equation}  
such that for all $g\in G$ and $i\in I $ we have $ g A_i g^{-1}= A_{\varepsilon(g)i } $,  where $A_i$ is the $i$-labeled copy of $A$ in the direct sum $A^{(I)}$.

\vskip 0.1in

In this paper, we prove the following result on bi-exactness.
\begin{maintheorem}\label{mainthm}
Let $G\in\mathcal{WR}(A,B\curvearrowright I)$ be a wreath-like product group such that $A$ is amenable, $B$ is bi-exact, and the action $B\curvearrowright I$ has amenable stabilizers. Then $G$ is bi-exact.
\end{maintheorem}

When the extension \eqref{wrlext} splits, so that $G$ is a classical generalized wreath product, Theorem~\ref{mainthm} recovers the corresponding bi-exactness results from \cite{BO08} and \cite[Theorem 4.7]{De19}. More importantly, the theorem applies to genuinely non-split wreath-like extensions and, in particular, to the property (T) W$^*$-superrigid groups introduced in \cite{CIOS1}. To the best of our knowledge, these provide the first examples of bi-exact property (T) W$^*$-superrigid groups. The theorem also provides a different route to several structural phenomena studied in \cite{AMCOS25}, where a uniform boundedness assumption on the $2$-cocycle defining the wreath-like extension was imposed in order to construct the deformation-theoretic tools used there.

We record two immediate consequences. First, by Ozawa's solidity theorem for bi-exact groups \cite{Oza04}, Theorem~\ref{mainthm} yields the following.

\begin{maincor}\label{solid-cor}
Let $G\in\mathcal{WR}(A,B\curvearrowright I)$ be as in Theorem~\ref{mainthm}. Then $L(G)$ is solid.
\end{maincor}

This considerably extends the solidity-type result obtained in \cite[Corollary 4.7]{CIOS1}. There the base group $A$ was assumed to be abelian and the argument relied on the structure of the associated equivalence relations. Corollary~\ref{solid-cor}, by contrast, applies to arbitrary amenable base groups and follows directly from the bi-exactness established in Theorem~\ref{mainthm}. 

A second application concerns the McDuff superrigidity phenomenon introduced in \cite[Definition 1.1]{AMCOS25}. The examples constructed there arise as infinite direct sums of property (T) W$^*$-superrigid wreath-like product groups whose defining $2$-cocycles satisfy a uniform boundedness condition \cite[Definition 23]{AMCOS25}. Theorem~\ref{mainthm} allows us to completely remove this restriction. Indeed, the wreath-like product groups covered by Theorem~\ref{mainthm} are bi-exact, and hence the infinite direct-sum rigidity theorem of Ding--Drimbe \cite[Theorem 6.2]{DD25}, combined with the W$^*$-superrigidity input used in \cite{AMCOS25}, gives the following.

\begin{maincor}\label{mcduff-cor}
Let $(G_n)_{n\in\mathbb N}$ be a sequence of property (T), W$^*$-superrigid groups with $G_n\in\mathcal{WR}(A_n,B_n\curvearrowright I_n)$, where $A_n$ is abelian, $B_n$ is any icc subgroup of a hyperbolic group, and the action $B_n\curvearrowright I_n$ has amenable stabilizers, for every $n\in\mathbb N$. Then the infinite direct sum
$\bigoplus_{n\in\mathbb N}G_n$
is McDuff superrigid. In particular, if $H$ is any countable group such that
$L(H)\cong L(\bigoplus_{n\in\mathbb N}G_n)$,
then one can find an icc amenable group $A$ such that
$H\cong (\bigoplus_{n\in\mathbb N}G_n)\times A$.
\end{maincor}

\vskip 0.1in 

 The proof of Theorem \ref{mainthm} builds on the strategy developed by Brown and Ozawa for classical wreath products in \cite[Section 15.3]{BO08}, and subsequently adapted by Deprez to the locally compact setting in \cite{De19}. A central ingredient in these arguments is the construction of suitable weighted measures on the support of an element, whose asymptotic equivariance eventually yields the probability measures required for bi-exactness. However, in the wreath-like setting considered here, the extension need not split, and multiplication on the left and right introduces nontrivial cocycle terms into the individual coordinates. As a result, the length-function arguments available for ordinary wreath products cannot be applied directly. To overcome this technical difficulty, we construct proper weights adapted to the finitely many coordinate transformations arising from these cocycles and establish the corresponding weighted-support estimates. Combined with suitable lifting and diagonalization arguments, this allows us to extend the original Brown--Ozawa and more recent Deprez method to the more general class of wreath-like groups considered in this paper.

In addition, we recast these arguments in terms of equivariant u.c.p.\ maps on the boundary bidual. This yields a more streamlined proof and, at the same time, highlights a different perspective on the underlying techniques by placing them within the framework of boundary actions.
\vskip 0.08in
\noindent{\bf Acknowledgments.} The first author was partially supported form the NSF grant DMS-2452247.

\noindent{\bf AI tool disclosure.} The authors used ChatGPT (OpenAI) for English-language editing, proofreading, and improving the organization and clarity of the exposition. It was also used to assist with literature searches for related results and references, as well as with developing, checking, and extending some of the mathematical arguments. The authors take full responsibility for the content and final form of the paper, including the accuracy of the references and the correctness of all mathematical statements and proofs.

\section{Proof of the Main Result}

To prove our main theorem, we first need two preliminary results. The first is a discrete, two-sided version of the diagonalization lemma \cite[Lemma 3.2]{De19}. We follow closely Deprez's argument, with minor adaptations to our setting. As observed in \cite{De19}, this argument is based on the diagonal trick from \cite[Exercise 15.1.1]{BO08}. We include the details only for reader's convenience.

\begin{lemma}\label{unif-riter} Let $G$ be a countable group. Assume that for every finite subset $K\subset G$ and every scalar $\delta>0$ there exists a map $\mu_{K,\delta}: G \longrightarrow \Prob(G)$ and a finite subset $D_{K,\delta}\subset G$ such that \[ \|\mu_{K, \delta}(kgh)- k\mu_{K,\delta}(g)\|_1<\delta,\]
for all $k,h \in K$ and $g\in G\setminus D_{K,\delta}$. Then  there exists a map $\mu: G \longrightarrow \Prob(G)$ such that 
\[\lim_{g\rightarrow \infty}\|\mu(kgh)-k\mu(g) \|_1=0. \]
    
\end{lemma}

\begin{proof} 
 Pick  an increasing sequence of finite, symmetric
sets $e\in K_1\subset K_2\subset\cdots \subset G$ such that  $\bigcup_{n\ge1}K_n=G$. For each $n$, pick a map $\mu_n:G\longrightarrow\Prob(G)$ together with a finite set
$D_n\subset G$ such that
\begin{equation}\label{mun-control}
  \|\mu_n(kgh)-k\mu_n(g)\|_1<\frac{1}{2^n},
\end{equation}
whenever $k,h\in K_n$ and $g\notin D_n$.

Next consider an increasing sequence of finite sets $L_n\subset G$ such that
$\bigcup_nL_n=G$, $D_n\subset L_n$, and
\begin{equation}\label{Ln-stability}
  K_nL_{n-1}K_n\subset L_n.
\end{equation}
This can be done recursively, while also inserting the first $n$ elements of
a fixed enumeration of $G$ into $L_n$.  

For any $g\in G$, let
\begin{equation}
  r(g)=\min\set{n\ge1\,|\,g\in L_n}.
\end{equation}
and notice that  $r(g)\rightarrow\infty$ as $g\rightarrow\infty$.

For $r\ge4$, let
  $I_r:=\set{\lfloor r/2\rfloor,\lfloor r/2\rfloor+1,\ldots,r-1}$ and let $a_r$ be the uniform probability measure on it.  Define
\begin{equation}\label{diagonal-mu}
  \mu(g):=\sum_{n\in I_{r(g)}}a_{r(g)}(n)\mu_n(g)
\end{equation}
when $r(g)\ge4$, and define $\mu(g)$ arbitrarily for the remaining finitely
many elements $g$.

Fix $k,h\in G$, and choose $m$ with $k,h\in K_m$.  To simplify the writing, let
$r=r(g)$ and $r'=r(kgh)$. Next we show that if  $r$ is sufficiently large, then we have
\begin{equation}\label{r-change}
  \abs{r'-r}\leq 1.
\end{equation}
Indeed, since $k,h\in K_{r+1}$ and $g\in L_r$, relation
\eqref{Ln-stability} with $n=r+1$ gives $kgh\in L_{r+1}$ and therefore $r'\le r+1$.
If we assume by contradiction that $r'\le r-2$, then $kgh\in L_{r-2}$. As
$k^{-1},h^{-1}\in K_{r-1}$, using \eqref{Ln-stability} again, we further see that
\begin{equation*}
  g=k^{-1}(kgh)h^{-1}\in K_{r-1}L_{r-2}K_{r-1}\subset L_{r-1},
\end{equation*}
which contradicts the definition of $r$. Altogether, these show \eqref{r-change}.

We now compare the two averages defining $\mu(kgh)$ and $k\mu(g)$. For $n\in I_r\cap I_{r'}$ and $r$ sufficiently large, we have $n\ge m$ and
$n\le r-1$.  Since $D_n\subset L_n\subset L_{r-1}$ while
$g\notin L_{r-1}$, estimate \eqref{mun-control} applies.  Consequently,
using that all $\mu_n(k)$ have norm one,
\begin{equation*}\begin{split}
  \|\mu(kgh)-k\mu(g)\|_1
  &\leq \|a_{r'}-a_r\|_1
     +\sum_{n\in I_r\cap I_{r'}}
       \min\{a_r(n),a_{r'}(n)\}
       \|\mu_n(kgh)-k\mu_n(g)\|_1 \\
  &\leq \|a_{r'}-a_r\|_1
     +2^{-\lfloor r/2\rfloor+1}.\end{split}
\end{equation*}
If $\abs{r'-r}\le1$, an elementary comparison of the two adjacent uniform
measures gives
\begin{equation*}
  \|a_{r'}-a_r\|_1\leq \frac{8}{r-1}.
\end{equation*}
The right-hand side therefore tends to zero as $g\to\infty$.  This gives the desired conclusion. \end{proof}
\vskip 0.06in
The second preliminary result is the discrete version of the uniformity argument used in \cite[Proposition 2.5 and the proof of Theorem 4.7, Step 2]{De19}. We follow the same argument, which in the discrete setting admits the following simple formulation, and include the details for completeness. Intuitively, this says that we can often ignore amenable subgroups when proving bi-exactness, see also the proof of \cite[Lemma 15.2.6]{BO08}.

\begin{lemma}
\label{uniform-coset-reiter}
Let $G$ be a countable group and let $P<G$ be amenable.  For every
finite set $F\subset G$ and every $\delta>0$, there exists a map $\eta:G/P\longrightarrow\Prob(G)$ such that for every $g\in F$ we have 
\begin{equation}\label{ucreit}
  \sup_{z\in G/P}\|g\eta(z)-\eta(gz)\|_1<\delta.
\end{equation}
\end{lemma}

\begin{proof} Replacing $F$ by $F\cup F^{-1}\cup\{1\}$, we may assume that $F$ is symmetric
and contains the identity.  Put $X=G/P$, choose a section
$s:X\rightarrow G$, and write
\begin{equation}\label{coset-cocycle}
  gs(z)=s(gz)c(g,z),\quad \text{ where } c(g,z)\in P.
\end{equation}

Consider finite sets $L_1\subset L_2\subset\cdots\subset X$ so that $\bigcup_{n\ge1}L_n=X$
such that, for every $n$, we have 
\begin{equation}\label{coset-exhaustion}
  FL_n\subset L_{n+1}.
\end{equation}  

Consider
  $r(z):=\min\set{n\ge1\,|\,z\in L_n}$. By symmetry of $F$, relation \eqref{coset-exhaustion} implies for all $g\in F$ and $z\in X$ we have  
\begin{equation}\label{coset-r-change}
  \abs{r(gz)-r(z)}\leq 1.
\end{equation}

Now fix $\rho>0$ and $m\in\mathbb N$.  For each $n$,
the set $C_n=\set{c(g,z)\,|\,g\in F,\ z\in L_n}\subset P$ is finite.  By amenability of $P$, pick $\mu_n\in\Prob(P)$ such that for all $c\in C_n$ we have 
\begin{equation}\label{Folner-mn}
  \|c \mu_n-\mu_n\|_1<\rho
  .
\end{equation}

Let $\nu_n(z):=s(z)\mu_n\in\Prob(G)$. If $z\in L_n$ and $g\in F$, then using \eqref{coset-cocycle} and
\eqref{Folner-mn} we get
\begin{equation}\label{local-coset-control}
  \|g\nu_n(z)-\nu_n(gz)\|_1
  =\|c(g,z)\mu_n-\mu_n\|_1<\rho.
\end{equation}

For $z\in X$, define
\begin{equation*}\label{eta-shell-average}
  \eta(z)=\frac{1}{m}\sum_{n=r(z)+1}^{r(z)+m}\nu_n(z).
\end{equation*}

Fix $g\in F$.  By \eqref{coset-r-change}, the two index intervals occurring
in $g\eta(z)$ and $\eta(gz)$ have symmetric difference of cardinality at most
two.  On every common index $n$ we have $z\in L_n$, and therefore, using
\eqref{local-coset-control} for every $g\in F$ and $z\in X$ we have
\begin{equation*}
  \|g\eta(z)-\eta(gz)\|_1\leq \rho+\frac{2}{m}
  .\end{equation*}
  
Taking $\rho<\delta/2$ and $m>4/\delta$, one gets the desired conclusion.\end{proof}

\vskip 0.08in

We now focus on wreath-like products. Throughout the remainder of this paper we let $G \in WR(A, B\curvearrowright I)$ as in the Theorem \ref{mainthm} and $N= A^{(I)}= \ker(\varepsilon)$.  Here, and in what follows, $\varepsilon: G \longrightarrow B$ is the canonical epimorphism defining the wreath-like extension $G$.  Since $A$ is amenable, it follows that $N$ is amenable. We fix a section $s:B\longrightarrow G$, so that $ \varepsilon\circ s = \text{id}_B $. The choice of section gives us a twisted cocycle $ \alpha(b,c):=s(b)s(c)s(bc)^{-1} $. For each $b\in B$, the section also induces a map on $N$: $ \sigma_b:=\text{Ad}(s(b))|_N $, as well as the map $ \sigma_{b,i}:=\sigma_b|_{A_i}:A_i\to A_{bi} $.  For $g=xs(b)\in G$, with $x=(x_i)_{i\in I}\in N=A^{(I)}$, we define the support of $g$ by
\(
{\bf s}(g):=\{\,i\in I\mid x_i\neq e\,\}.
\)

Since $B$ is bi-exact, one can find a map $\mu : B \longrightarrow \Prob(B)$ such that 

\begin{equation}\label{biexact1}\lim_{g \rightarrow \infty}\|\mu(kgh)-k\mu(g)\|_1=0,\end{equation}
for all $k,h\in B$.
\vskip 0.1in

We continue with the following lemma using the bi-exactness of $B$, which is a discrete version of the lifting argument from \cite[Lemma 4.3]{De19}. In our setting, the argument admits a particularly simple formulation using the preceding lemma and a barycenter construction, which we include for completeness.
 
\begin{lemma}\label{quotientlift1} For every finite set $K\subset G$ and scalar $\delta>0$ there is a map $Q: B\longrightarrow \Prob(G) $ and a finite set $D\subset B$ such that for all $k,h \in K$ and $b\in B\setminus D$ we have \begin{equation}
\|Q(\varepsilon(k) b \varepsilon(h))- k Q(b)\|_1<\delta.
\end{equation}
  \begin{proof} Identify $B = G/N$. Since $N$ is amenable, by Lemma \ref{uniform-coset-reiter} there is a map $\eta: B \longrightarrow \Prob (G)$  such that for all $k\in K$ we have
  \begin{equation}\label{control1}
     \sup_{b\in B}\|k \eta(b)- \eta(\varepsilon(k)b)\|_1<\frac{\delta}{2}. 
  \end{equation}
Using \eqref{biexact1} one can find  a finite set $D\subset B$ such that  \begin{equation}\label{control-2}
    \|\mu(\varepsilon(k)b \varepsilon(h))-\varepsilon(k)\mu(b)\|_1<\frac{\delta}{2},
\end{equation}
for all $k,h\in K$ and $b\in B\setminus D$.

Now consider the barycenter map $Q: B \longrightarrow \Prob(G)$ given by 
\begin{equation*}
    Q(b)=\sum_{c\in B} \mu(b)(c)\eta(c).   
\end{equation*}
Using the definitions, the contractivity property, and inequalities \eqref{control1}-\eqref{control-2} we can see that 
\begin{equation*}\begin{split}
    \|Q(\varepsilon(k) b \varepsilon(h))- k Q(b)\|_1&\leq \|\mu(\varepsilon(k)b\varepsilon(h))- \varepsilon(k)\mu(b)\|_1+\sum_{c\in B} \mu (b)(c)\|k \eta(c)-\eta(\varepsilon(k)c) \|_1\\
    & \leq \frac{\delta}{2}+\frac{\delta}{2}=\delta,\end{split}
\end{equation*}
which yields the result.  \end{proof}  
\end{lemma}

As a corollary, this in fact shows that $G $ is bi-exact relative to $N$. Recall that, if $H<G$ is a subgroup, we say that $G$ is bi-exact relative to $H$ if $G$ is exact and there exists a map $
\mu:G\to \Prob(G)
$ such that, for every $g,h\in G$,
$$
\lim_{x\to\infty/H}
\|\mu(gxh)-g\cdot\mu(x)\|_1=0,
$$ where $x\to\infty/H$ means that $x$ eventually leaves every subset small relative to $H$. Here a subset of $G$ is small relative to $H$ if it is contained in a finite union of double translates $sHt$, $s,t\in G$. For our case, we have $x\to \infty/N $ iff $ \varepsilon(x)\to \infty $ in $B$.

\begin{cor}
    If $ A $ is amenable, and $B$ is bi-exact, then $G$ is bi-exact relative to $ N $.
\end{cor}
\begin{proof}
    Since $N$ is amenable and $B$ is exact, $G$ is also exact. Fix a finite subset $K\subset G$ and $\delta>0$. By Lemma~\ref{quotientlift1}, there exists a map $ Q_{K,\delta} : B\to \text{Prob}(G) $ and a finite subset $ D_{K,\delta}\subset B $ such that
    \[ \| Q_{K,\delta}(\varepsilon(k)b\varepsilon(h))-kQ_{K,\delta}(b) \|_1<\delta,\quad \text{ for all }k,h\in K,\;b\in B\setminus D_{K,\delta}.  \]
    The same diagonal argument as in Lemma~\ref{unif-riter}, yields a map $Q:B\to \text{Prob}(G)$ such that
    \[ \lim_{ b\to \infty}\|Q( \varepsilon(k)b\varepsilon(h))-kQ(b) \|_1=0.\]
    Now, define $ \mu: G\to \text{Prob}(G) $, $\mu(g):=Q(\varepsilon(g))$. Then
     \begin{align*}
\|\mu(kgh)-k\mu(g)\|_1
&=
\|
Q(\varepsilon(k)\varepsilon(g)\varepsilon(h))
-kQ(\varepsilon(g))
\|_1 \end{align*}
which tends to $0$ as $g \to \infty/N$, since $ g\to \infty/N $ implies $ \varepsilon(g)\to \infty $.
\end{proof}
\vskip 0.05in
Alternatively, the corollary follows directly from the array characterization of relative bi-exactness. Indeed, as observed in \cite[Example 2.7(A)]{CSU11}, a proper array $r:B\to\ell^2(B)$ pulls back through $\varepsilon$ to an array $\widetilde r(g)=r(\varepsilon(g))$ for all $g\in G$, with values in $\ell^2(G/N)$, which is proper relative to $N$. Since $N$ is amenable, the quasi-regular representation $\lambda_{G/N}$ is weakly contained in $\lambda_G$. The conclusion therefore also follows from the array characterization of relative bi-exactness from \cite[Remark 2.4]{CSU11} and \cite[Proposition 2.7]{PV12}.

\vskip 0.1in

The remainder of the proof builds on the weight strategy for wreath products originating in \cite[Corollary 15.3.6]{BO08} and developed further in \cite[Proposition 4.4 and Theorem 4.7]{De19}. In the present wreath-like setting, however, the presence of cocycle terms arising from the extension makes the coordinate transformations more technically involved, so the usual length-function argument does not apply directly. We develop below all necessary modifications needed in this setting.

\begin{lemma}
\label{properness1}
Let $X$ be a countable set and let $\mathcal T$ be a finite family of
bijections of $X$.  Then one can find a proper map $\rho:X\longrightarrow\mathbb N $
such that for every $T\in\mathcal T$ and $\xi\in X$ we have 
\begin{equation}\label{eq:slow-weight}
  \abs{\rho(T\xi)-\rho(\xi)}\le1.
\end{equation}
Here, proper means
$\set{\xi\in X\,|\,\rho(\xi)\le R}$ is finite for every $R\ge 0$.
\end{lemma}

\begin{proof}
Consider the finite symmetric family $\mathcal S:=\mathcal T\cup\mathcal T^{-1}\cup\{\id_X\}$. Enumerate $X=\{\xi_1,\xi_2,\ldots\}$.  Next we construct recursively increasing finite sets $F_n\subset X$ with
   $\xi_1,\ldots,\xi_n\in F_n$ and $U(F_n)\subset F_{n+1}$ for every $U\in\mathcal S$.
For example, after choosing $F_n$, take
$F_{n+1}=F_n\cup\{\xi_{n+1}\}
          \cup\bigcup_{U\in\mathcal S}U(F_n)$. 

Let $\rho(\xi):=\min\set{n\ge1 \,|\,\xi\in F_n}.$
Notice that $\set{\xi\,|\,\rho(\xi)\le n}=F_n$, which entails $\rho$ is proper.  If
$T\in\mathcal T$ and $\rho(\xi)=n$, then $T\xi\in F_{n+1}$ and hence
$\rho(T\xi)\le n+1$.  Using the same argument for $T^{-1}$ at $T\xi$ yields
$\rho(\xi)\le\rho(T\xi)+1$.  Altogether, these yield \eqref{eq:slow-weight}.
\end{proof}

\vskip 0.04in
We continue our proof by introducing more notation. Fix finite sets
$K\subset G$ and $L\subset B$ and assume that $L\ne\emptyset$. Consider the coordinate space
\begin{equation}\label{Omega-def}
  \Omega=\bigsqcup_{i\in I}A_i
  =\set{(i,z)\,|\,i\in I,\ z\in A_i}.
\end{equation}

Fix $k,h\in K$,  $b\in L$ and let $k=a s(p)$, $g=xs(b)$, and $h=c s(q)$
where $p=\eps(k)$ and $q=\eps(h)$.  
Then we see that 
 $kgh=y s(pbq)$, where 
\begin{equation}\label{kgh}\begin{split}y&=a\,\sigma_p(x)\,r_{k,h,b}, \text{ and}\\r_{k,h,b}&=\alpha(p,b)\,\sigma_{pb}(c)\,\alpha(pb,q)\in N.\end{split}
\end{equation}
In particular, for every $i\in I$ we have  $y_{pi}=a_{pi}\,\sigma_{p,i}(x_i)\,(r_{k,h,b})_{pi}$. This induces the bijection $ T_{k,h,b}:\Omega\to \Omega $, $ T_{k,h,b}(i,z):=(pi, a_{pi}\,\sigma_{p,i}(z)\,(r_{k,h,b})_{pi} ) $.

Using Lemma \ref{properness1} for the finite collection of bijections  $\mathcal T_{K,L}
=\set{T_{k,h,b}\,|\,k,h\in K,\ b\in L}$ one can find  a proper map $\rho:=\rho_{K,L}:\Omega\longrightarrow\mathbb N$
such that for all $T\in \mathcal T_{K,L}
$ and $\omega \in \Omega$ 
we have

\begin{equation}\label{rho-cont}
  |\rho(T\omega)-\rho(\omega)|\leq 1.
\end{equation}

For each $g=xs(b)\in G$, define its extended support by  $\widetilde{\bf s}(g)=\set{(i,x_i)\,|\,i\in I,\ x_i\ne e_i}\subset\Omega$. Note that $\abs{\widetilde{\bf s}(g)}=\abs{{\bf s}(g)}$.  Consider the following finite positive measure on
$I$
\begin{equation}\label{base-sum}
  Z_g=\sum_{i\in {\bf s}(g)}\rho(i,x_i)\delta_i,
\end{equation}
and note that 
$\|Z_g\|_1
  =\sum_{i\in {\bf s}(g)}\rho(i,x_i)$.
\vskip 0.05in

\begin{lemma}\label{comparison-fin}
There is a constant $C_{K,L}<\infty$ such that
\begin{equation}\label{comparison}
  \|Z_{kgh}-\eps(k)Z_g\|_1
  \le \abs{{\bf s}(g)}+C_{K,L}
\end{equation}
for all $k,h\in K$ and all $g\in G$ satisfying $\eps(g)\in L$.
\end{lemma}

\begin{proof}
Fix $k,h\in K$, $b\in L$, and let $T=T_{k,h,b}$.  Let
$p=\eps(k)$.  For $g=xs(b)$ and $kgh=y s(pbq)$, equation
\eqref{kgh} yields that for all $i\in I$ we have $T(i,x_i)=(pi,y_{pi})$, where $y_{pi}=a_{pi}\,\sigma_{p,i}(x_i)\,(r_{k,h,b})_{pi}$.

Consider the set $E_T=\set{i\in I\,|\,T(i,e_i)\ne(pi,e_{pi})}$  and notice it is finite.  
  Indeed, since 
  $T(i,e_i)=\left(pi,a_{pi}(r_{k,h,b})_{pi}\right)$ then 
 $E_T\subset p^{-1}( {\bf s}(a)\cup {\bf s}({r_{k,h,b}}))$.

If $i\notin E_T$, then $T(i,e_i)=(pi,e_{pi})$.  Since the restriction of $T$
to the fiber $A_i$ is a bijection onto the fiber $A_{pi}$, we have $x_i=e_i$ if and only if $ y_{pi}=e_{pi}$.

Whenever these two coordinates are nontrivial, \eqref{rho-cont} yields
\begin{equation}\label{coordinate-diff}
  \abs{\rho(pi,y_{pi})-\rho(i,x_i)}\le1.
\end{equation}

For $i\in E_T$, there are three possibilities which we analyze individually.  

If $x_i\neq e_i$ and $y_{pi}\neq e_{pi}$
are nontrivial, \eqref{coordinate-diff} still
holds.  

If $x_i=e_i$ and $y_{pi}\ne e_{pi}$, the contribution at the
coordinate $pi$ equals  $\rho(T(i,e_i))$.  

If
$x_i\ne e_i$ and $y_{pi}=e_{pi}$, then
$(i,x_i)=T^{-1}(pi,e_{pi})$ and hence
the contribution equals 
$\rho(T^{-1}(pi,e_{pi}))$.

As $i\rightarrow pi$ is a bijection, the preceding
case analysis shows that
\begin{equation}\label{eq:C-T-bound}
  \|Z_{kgh}-pZ_g\|_1\le \abs{{\bf s}(g)}+C_T,
\end{equation}
where $C_T:=
  \sum_{i\in E_T}\left(1+\rho(T(i,e_i))+\rho(T^{-1}(pi,e_{pi}))
  \right)$.

Since $\mathcal T_{K,L}$ is a finite family then $C_{K,L}:=\max_{T\in\mathcal T_{K,L}}C_T<\infty$
 and works for all $k,h\in K$ and $b\in L$.
\end{proof}

\begin{lemma}\label{properness}
For every finite $K\subset G$, $L\subset B$ we have 
\begin{equation}\label{sup-over-W}
  \lim_{g \to \infty, \varepsilon(g)\in L }\frac{\abs{{\bf s}(g)}}{\|Z_g\|_1}=0.
\end{equation}
\end{lemma}

\begin{proof}
Recall that $\rho$ depends on the fixed finite sets $K$ and $L$.  

First we claim that \begin{equation}\label{lim-inf}\lim_{g \to\infty, \varepsilon(g)\in L }\| Z_g\|_1=\infty.\end{equation} 

To see this, fix $R>0$. Since $\rho$ is proper, the set  $\Omega_R=\set{\omega\in\Omega\,|\,\rho(\omega)\le R}$ is 
finite.

\noindent Now suppose that $\eps(g)\in L$ and $\|Z_g\|_1\le R$.  Thus $\widetilde{\rm s}_g\subset\Omega_R$ and since $L$ is finite,
there are only finitely many elements $g=xs(b)$ satisfying
$\eps(g)\in L$ and $\|Z_g\|_1\leq R$.  This proves \eqref{lim-inf}.
\vskip 0.03in
To see \eqref{sup-over-W}, fix $R\ge1$.  Since at most $\abs{\Omega_R}$ elements 
$\omega\in \widetilde{\bf s} (g)$ satisfy  $\rho(\omega)\le R$ we have   
\begin{equation*}
  \|Z_g\|_1\ge R(\abs{{\bf s}_g}-\abs{\Omega_R}),
\end{equation*}
and hence
\begin{equation}\label{ratio-bound}
  \frac{\abs{{\bf s} (g)}}{\|Z_g\|_1}
  \le \frac1R+\frac{\abs{\Omega_R}}{\|Z_g\|_1}.
\end{equation}
Letting $g\to\infty$ with $\eps(g)\in L$ and then  $R\to\infty$ in the inequality \eqref{ratio-bound}, we get
\eqref{sup-over-W}.
\end{proof}

\begin{prop}
\label{nu-comparison} Now fix a basepoint  $i_0\in I$ and we let  $\widehat Z_g=Z_g+\delta_{i_0}$. Then the map $\nu:G\to\Prob(I)$ given by $\nu_g=\frac{\widehat Z_g}{\|\widehat Z_g\|_1}$
for all $g\in G$ satisfies the following
\begin{equation}\label{nu-asymptotic}
  \lim_{g \to \infty, \varepsilon(g)\in L}
  \|\nu_{kgh}-\eps(k)\nu_g\|_1
  =0
\end{equation}
for all $k,h\in K$.
\end{prop}

\begin{proof}
Fix $k,h\in K$ and put $p=\eps(k)$.  Since $\widehat Z_{kgh}=Z_{kgh}+\delta_{i_0}$, $p\widehat Z_g=pZ_g+\delta_{pi_0}$, Lemma~\ref{comparison-fin} gives
\begin{equation}\label{comparison2}
  \|\widehat Z_{kgh}-p\widehat Z_g\|_1
  \le \abs{{\bf s}(g)}+C_{K,L}+2.
\end{equation}

Using \eqref{comparison2} together with the elementary inequality $\|
    \frac{\lambda}{\|\lambda\|_1}
    -\frac{\zeta}{\|\zeta\|_1}
  \|_1
  \le
  \frac{2\|\lambda-\zeta\|_1}{\|\zeta\|_1}$ for the positive measures $\lambda = \widehat Z_{kgh}$ and $\zeta=p\widehat Z_g$ we obtain
\begin{equation*}
  \|\nu_{kgh}-p\nu_g\|_1
\leq\frac{2(\abs{{\bf s}(g)}+C_{K,L}+2)}{\|Z_g\|_1+1}.
\end{equation*}
Finally, using
Lemma~\ref{properness} above, we get the desired conclusion.
\end{proof}

\begin{prop}\label{supportlift2} For every finite subsets $K\subset G$ and $L\subset B$ and every $\delta>0$ there is  map $M: G \longrightarrow \Prob(G)$ and a finite set $D_M\subset G$ such that 

\begin{equation}
    \|M(kgh)-kM(g)\|_1<\delta,
\end{equation}
for all $k,h\in K$ and $g\in G\setminus D_M$ such that $\varepsilon(g)\in L$.   
\end{prop}
\begin{proof} First, decompose $ I  $ into $B$-orbits, $ I = \bigsqcup_\beta I_\alpha \simeq \bigsqcup_\beta B/H_\beta $. Let $P_\beta=\varepsilon^{-1}(H_\beta)<G$ and notice that $P_\beta$ is amenable for each $\beta$. Moreover, we have that $G/P_\beta=B/H_\beta=I_\beta$.
Using Lemma \ref{uniform-coset-reiter} for $P_\beta<G$ one can find a map $\eta_\beta : I_\beta \to \Prob(G)$ such that \begin{equation}\label{ineq2}
    \sup_{i\in I_\beta}\|k\eta_\beta (i)-\eta_\beta(\varepsilon(k) i) \|_1<\frac{\delta}{2}.
\end{equation}
Put these $\eta_\beta$'s together, we obtain a map $\eta:I \to \Prob(G)$, $\eta(i):=\eta_\beta(i)$ for $i\in I_\beta$.
Now define \[M(g)=\sum_{i\in I} \nu_g(i)\eta(i).\]
Using Proposition \ref{nu-comparison} one can find a finite set $D_M \subset G$ such that for all $g\in G \setminus D_M$ with $\varepsilon(g)\in L$ and $k,h\in K$ we have 
\begin{equation}\label{ineq3}
    \| \nu_{kgh}-\varepsilon(k)\nu_g\|_1\leq \frac{\delta}{2}.
\end{equation}
Letting $p=\varepsilon(k)$, and using basic estimates and \eqref{ineq2} and \eqref{ineq3} see that 
\begin{equation}\begin{split}\|M(kgh)-kM(g)\|_1&\leq \|\sum_i \nu_{kgh}(i)\eta(i)-\sum_i(p\nu_g)(i) \eta(i)\|_1+\|\sum_i \nu_g(i)\eta(pi)-\sum_i\nu_g(i)k\eta(i)\|_1\\& \leq \|\nu_{kgh}-p\nu_g\|_1+\sup_i\|\eta(pi)-k\eta(i)\|\\&\leq  \frac{\delta}{2}+\frac{\delta}{2}=\delta,\end{split}
\end{equation}
as desired.
\end{proof}

\vskip 0.08in
With these preliminaries at hand we are now ready to prove our main result.

\noindent \emph{Proof of Theorem \ref{mainthm}.} We will first show that the assumption in Lemma \ref{unif-riter} holds. Fix $K\subset G$ finite and $\epsilon>0$. Also let $\delta>0$ and $m\in \mathbb N$ such that
\begin{equation}\label{treshhold}
    2\delta +\frac{2}{m}<\epsilon.
\end{equation}

By Lemma \ref{quotientlift1} there is $Q: B\longrightarrow \Prob(G)$ and a finite set $D_Q\subset B$ such that  for all $k,h\in K$ and all $b\in B\setminus D_Q$ we have \begin{equation}\label{Q-control}
    \|Q(\varepsilon(k)b\varepsilon(h))-kQ(b)\|_1<\delta.
\end{equation} 
Now consider the finite set \begin{equation}
    C= D_Q \cup \bigcup_{k,h\in K}\varepsilon(k)D_Q\varepsilon(h)\subset B
\end{equation}

On $B$ consider the following (unoriented) graph structure: any $b\in B$ is connected to any $\varepsilon(k)b\varepsilon(h)$ for all $k,h\in K$. Then consider the metric given as follows. For any $b,c\in B$ let $d(b,c)$ to be the smallest integer $n$ for which there exist elements such that $b=b_0,b_1,\ldots,b_n=c$, where $b_j$ is adjacent to $b_{j+1}$ in the aforementioned graph structure on $B$. Also we let $d(b,c)=+\infty$ whenever $b$ and $c$ are in disjoint components.

Next consider the map $\chi: B \rightarrow [0,1]$ given by 

\begin{equation}
    \chi(b)=1- \frac{\min\{(\min_{c\in C}d(b,c)),m\}}{m},
\end{equation}
and notice that for all $k,h\in K$ we have \begin{equation}\label{chi-control}
  |\chi(\varepsilon(k)b \varepsilon(h))-\chi(b)|\leq \frac{1}{m}.  
\end{equation} 
Also, since $C$ is finite and the metric $d$ is locally finite it follows that the set $L=\{b\in B\,|\, d(b,C)\leq m\}$ is finite. Using Proposition \ref{supportlift2} for $K$, $L$ and $\delta>0$ one can find a map $M: G \longrightarrow \Prob (G)$ such that and a finite set $D_M \subset G$ such that for all $k,h\in K$ and $g\in G\setminus D_M$ with $\varepsilon(g)\in L$ we have  \begin{equation}\label{N-control}
    \|M(kgh)-kM(g)\|_1<\delta.
\end{equation}

Now define the following averaging map $\mu: G \longrightarrow \Prob(G)$ by 

\begin{equation}
    \mu(g)= \chi(\varepsilon(g)) M(g)+ (1-\chi(\varepsilon(g)))Q(\varepsilon(g))
\end{equation}

Now fix $k,h\in K$. Then for all $g\in G\setminus D_M$, we can see that 
\begin{equation}\label{globalmapcontrol}\begin{split}
    \|\mu(kgh)-k\mu(g)\|_1&\leq \chi(\varepsilon(kgh)) \|M(kgh)-kM(g)\|_1+(1-\chi(\varepsilon(kgh)))\|Q(\varepsilon(kgh))-kQ(\varepsilon(g))\|_1 +\\&+2 |\chi(\varepsilon(kgh))-\chi(\varepsilon(g)) |\end{split}
\end{equation}
We will control these terms separately.

If $\chi(\varepsilon(kgh))>0$ then $d(\varepsilon(kgh),C)<m$. As $\varepsilon(kgh)$ and $\varepsilon(g)$ are adjacent we get that $d(\varepsilon(g),C )\leq m$
 and hence $\varepsilon
 (g)\in L$. Thus by \eqref{N-control} we get that the first term of \eqref{globalmapcontrol} is smaller than $\delta$.

 If $1-\chi(\varepsilon(kgh))>0$ then $\chi(\varepsilon(kgh))<1$ and hence $\varepsilon(kgh)\nin C$.  If $\varepsilon(g)\in D_Q$ then $\varepsilon(kgh)\in C$, which is a contradiction. Hence $\varepsilon(g)\nin D_Q$ and therefore by \eqref{Q-control}, the second term of \eqref{globalmapcontrol} is smaller than $\delta$.

 Finally, by \eqref{chi-control} the third term of \eqref{globalmapcontrol} is at most $\frac{2}{m}$. Thus using these estimates together with \eqref{treshhold}
we get that for all $k,h\in K$ and $g\in G\setminus D_M$ with $\varepsilon (g)\in B\setminus D_Q$ we have that 
\begin{equation*}
   \|\mu(kgh)-k\mu(g)\|_1 \leq 2\delta + \frac{2}{m}<\epsilon.
\end{equation*}

Thus Lemma \ref{unif-riter} implies that there is a map $\mu:G \longrightarrow \Prob(G)$ such that for all $k,h\in G$ we have 

\begin{equation}\label{S}\lim_{g\rightarrow \infty}\|\mu(kgh)-k\mu(g)\|_1=0.\end{equation}

The kernel $A^{(I)}$ is amenable, hence exact.  The quotient $B$ is bi-exact,
hence exact.  Exactness of countable discrete groups is stable under group
extensions.  Therefore $G$ is exact. Combined with \eqref{S}, this shows $G$ is bi-exact. $\hfill\qed$

\subsection{A bidual interpretation of the proof}

We conclude this section by reformulating the preceding argument in terms of equivariant u.c.p. maps on the boundary bidual. This does not give a separate proof of Theorem~\ref{mainthm}, but makes the two parts of the argument more transparent.

\subsubsection{The relative bi-exactness with respect to $N$ part}

We use $\rho_h$ to denote the right translation action of $h\in G$ on
$\ell^\infty(G)$. Recall that bi-exactness
of $G$ is equivalent to amenability of the left $G$-action on the Higson corona algebra $
\bigl(\ell^\infty(G)/c_0(G)\bigr)^{G_r}$,
 where $S^{G_r}$ means the right $G$-invariant part of $S$.
For exact groups, this is also equivalent to having a $G$-equivariant u.c.p. (unital completely positive) map 
\[
\ell^\infty(G)\longrightarrow\left[
\bigl(\ell^\infty(G)/c_0(G)\bigr)^{**}
\right]^{G_r},
\]
see \cite[Theorem~5.5 (3)]{KEY26}. We will use this bidual formulation to
reinterpret the two parts of the preceding proof as maps into
complementary central corners.

At the level of u.c.p.\ maps, Lemma~\ref{uniform-coset-reiter} says that amenability of $N=A^{(I)}$ yields a left $G$-equivariant u.c.p. map $$ E_N:\ell^\infty(G)\longrightarrow \ell^\infty(B) = \ell^{\infty}(G/N) $$ by averaging functions on $G$ over $N$ with respect to an invariant mean on $N$.

Lemma~\ref{quotientlift1} can be viewed as the composition of two parts. First, since $B$ is bi-exact, there is a $B$-equivariant u.c.p. map $$ \Phi_B: \ell^\infty(B) \longrightarrow [(\ell^\infty(B)/c_0(B))^{**}]^{B_r}. $$ 

The second map is a canonical map \[ \iota^{**}_\varepsilon:[(\ell^\infty(B)/c_0(B))^{**}]^{B_r}\longrightarrow q_{N}^{\perp}[(\ell^\infty(G)/c_0(G))^{**}]^{G_r}, \]
where $ q_{N} $ is the identity of $ c_0(G,N)^{**}\subset \ell^\infty(G)^{**} $ with $ c_0(G,N) $ the algebra of functions $f$ such that $ \{g\in G:|f(g)|>\delta \}$ is always contained in finitely many fibers $ Ns(b) $ for all $\delta>0$. Indeed, consider the pullback $ \iota_\varepsilon:\ell^\infty(B)\longrightarrow \ell^\infty(G),
\;
\iota_\varepsilon(f)=f\circ\varepsilon $, and let $\iota^{**}_\varepsilon$ be its bidual map multiplied by $q_N^{\perp}$. Note that $q_N = \vee_{b\in B}\iota(1_{Ns(b) } ) $, where $\iota: \ell^\infty(G)\longrightarrow \ell^\infty(G)^{**}$ is the canonical embedding. Therefore if $ f\in c_0(B) $, then since $  (f\circ \varepsilon)1_{Ns(b)} = f(b)1_{Ns(b)} $, $ (\sum_{b\in F} 1_{Ns(b)} )_{F\Subset B}$ is an approximate identity for $ \iota_\varepsilon(f)=f\circ \varepsilon $. This implies that $ \iota(\iota_\varepsilon(f))q_{N}^{\perp}=0 $, so $ \iota^{**}_\varepsilon $ is well defined on $ (\ell^\infty(B)/c_0(B))^{**} $. It remains to check that the image of a $B_r$ invariant element under $\iota_\varepsilon^{**}$ is $G_r$ invariant: The function $ \iota_\varepsilon(f)=f\circ \varepsilon $ satisfies $ 
\rho_h^G(f\circ \varepsilon) = (\rho_{\varepsilon(h)}^{B}f)\circ \varepsilon$. This implies that $ \rho^G_h\circ \iota^{**}_\varepsilon= \iota^{**}_\varepsilon \circ \rho_{\varepsilon(h)}^B $, which shows the claim.

So the u.c.p. diagram for Lemma~\ref{uniform-coset-reiter} and Lemma~\ref{quotientlift1} is the $G$-equivariant u.c.p. map
\[ \Phi_{1}:\ell^\infty(G)\xrightarrow[A\text{ amenable}]{E_N} \ell^\infty(B) \xrightarrow[B \text{ bi-exact}]{ \Phi_B}  [(\ell^\infty(B)/c_0(B))^{**}]^{B_r}\xrightarrow{\iota_{\varepsilon}^{**}} q_{N}^{\perp}[(\ell^\infty(G)/c_0(G))^{**}]^{G_r} \]
which, by \cite[Theorem~5.5 (3)]{KEY26} and the exactness of $G$, is equivalent to $ G $ being bi-exact relative to $N$.

\subsubsection{Upgrading relative bi-exactness to full bi-exactness}

Similar to the construction of $E_N$, if $ B\ca I $ has amenable stabilizers, and $A$ is amenable, we can orbitwise define the $G$-equivariant u.c.p. map
\[ E_{ I }:\ell^\infty(G)\longrightarrow \ell^{\infty}(I). \]

Now, Proposition~\ref{nu-comparison} can be used to construct a u.c.p. map
\[ \Theta_\nu: \ell^\infty(I)\longrightarrow q_N[( \ell^\infty(G)/c_0(G) )^{**}]^{G_r}.  \]
More precisely, for all finite subsets $K\subset G$ and $L\subset B$, let $\nu_{K,L}:G\to \text{Prob}(I)$ be the map constructed in Proposition~\ref{nu-comparison}, and define the map $ \Theta_{ K,L }: \ell^\infty(I)\to \ell^\infty(G) $, $  \Theta_{ K,L }(f)(g):= \langle f, \nu_{K,L}(g)\rangle $. Let $ \Theta_\nu $ be a point-weak$^*$ cluster point of $ q_N\iota\circ \Theta_{K,L}: \ell^\infty(I)\to q_N( \ell^\infty(G)/c_0(G) )^{**} $ as $K\to G$ and $L\to B$.

To see the equivariance properties of $\Theta_\nu$, fix $b\in B$, $k,h\in G$, and $f\in\ell^\infty(I)$. Once $K$ contains $e,k^{-1},h$ and $L$ contains $b$, Proposition~\ref{nu-comparison} gives, as $g\to\infty$ with $\varepsilon(g)=b$, \[ \begin{aligned} \bigl| (k\Theta_{K,L}(f))(g) - \Theta_{K,L}(\varepsilon(k)\cdot f)(g) \bigr| &\leq \|f\|_\infty \bigl\| \nu_{K,L}(k^{-1}g) - \varepsilon(k)^{-1}\nu_{K,L}(g) \bigr\|_1 \longrightarrow0, \end{aligned} \] and \[ \begin{aligned} \bigl| \rho_h(\Theta_{K,L}(f))(g) - \Theta_{K,L}(f)(g) \bigr| &\leq \|f\|_\infty \bigl\| \nu_{K,L}(gh)-\nu_{K,L}(g) \bigr\|_1 \longrightarrow0. \end{aligned} \]
Since $ 
q_N=\bigvee_{b\in B}\iota(1_{Ns(b)})
$, passing to the weak$^*$ cluster point we obtain that $\Theta_\nu$ is $G$-equivariant and the image of $ \Theta_\nu $ is contained in $ q_N[( \ell^\infty(G)/c_0(G) )^{**}]^{G_r}$.

Composition $E_I$ and $\Theta_\nu$ we get the $G$-equivariant u.c.p. map
\[ \Phi_2:  \ell^\infty(G) \xrightarrow[\substack{A \text{ amenable}\\ B\ca I \text{ amenable stabilizers}}]{E_I} \ell^\infty(I) \xrightarrow[]{\Theta_\nu} q_N[ (\ell^\infty(G)/c_0(G))^{**} ]^{G_r}. \]

Adding up $\Phi_1$ and $\Phi_2$ we obtain the u.c.p. map for bi-exactness \[ \Phi=\Phi_1+\Phi_2: \ell^\infty(G)\ \to  [ (\ell^\infty(G)/c_0(G))^{**} ]^{G_r}.\]
Note that comparing to the original proof, the functions $\chi$ and $1-\chi$ become the indicator function of $q_N$ and $q_N^{\perp}$ in the limit. So overall the picture is
\begin{center}
\begin{adjustbox}{max width=\linewidth}
\begin{tikzcd}[
    column sep=2em,
    row sep=2.2em
]
&
\ell^\infty(B)
\arrow[r, "\Phi_B", "B\text{ bi-exact}"']
&
\bigl[(\ell^\infty(B)/c_0(B))^{**}\bigr]^{B_r}
\arrow[r, "\iota_\varepsilon^{**}"]
&
q_N^\perp
\bigl[(\ell^\infty(G)/c_0(G))^{**}\bigr]^{G_r}
\arrow[dr, hook]
&
\\
\ell^\infty(G)
\arrow[ur, "E_N", "A\text{ amenable}"']
\arrow[dr, "E_I",
"{\substack{
A\text{ amenable}\\
B\curvearrowright I\text{ has}\\
\text{amenable stabilizers}
}}"']
&&&&
\bigl[(\ell^\infty(G)/c_0(G))^{**}\bigr]^{G_r}
\\
&
\ell^\infty(I)
\arrow[rr, "\Theta_\nu"']
&&
q_N
\bigl[(\ell^\infty(G)/c_0(G))^{**}\bigr]^{G_r}
\arrow[ur, hook]
&
\end{tikzcd}
\end{adjustbox}

\[
\Phi_1=\iota_\varepsilon^{**}\circ \Phi_B\circ E_N,
\qquad
\Phi_2=\Theta_\nu \circ E_I,
\qquad
\boxed{\Phi=\Phi_1+\Phi_2}.
\]
\end{center}


\end{document}